\def\titre{Locally optimal controllers and globally inverse optimal controllers}
\def\longversion{1}

\if\longversion1
\documentclass[preprint,12pt]{article}
\else
\documentclass[preprint,12pt]{article}
\usepackage{natbib}
\fi

\usepackage{url}
\usepackage{graphicx}
\usepackage{color}
\usepackage{amsfonts,latexsym,amssymb,amsmath}

\makeatletter
\let\theoremstyle\@undefined                        
\makeatother
\usepackage{amsthm}

\newtheorem{theorem}{Theorem}

\newtheorem{assumption}{Assumption}

\def\RR{{\mathbb R}} 
\def\NN{{\mathbb N}} 
\def\HR{{\mathcal H}} 

\def\XR{\chi} 
\def\AR{\mathbb A} 
\def\BR{\mathbb B} 

\def\Ko{K_o}                

\def\alphao{\alpha_{o}}    

\DeclareMathOperator{\Id}{Id}         

\def\chi{{\mathchoice
{{\mbox{$\scriptstyle \mathcal{X}$}}}
{{\mbox{$\scriptstyle \mathcal{X}$}}}
{{\mbox{$\scriptscriptstyle \mathcal{X}$}}}
{{\mbox{$\scriptscriptstyle \mathcal{X}$}}}
}}

\catcode`\@=11
\def\downparenfill{$\m@th\braceld\leaders\vrule\hfill\bracerd$}
\def\overparen#1{\mathop{\vbox{\ialign{##\crcr\crcr \noalign{\kern0.4ex}
\downparenfill\crcr\noalign{\kern0.4ex\nointerlineskip}
$\hfil\displaystyle{#1}\hfil$\crcr}}}\limits}
\catcode`\@=12

\newcounter{comment}
\newlength{\comwidth}

\usepackage{authblk}
\begin{document}

\title{\titre}
\author[1]{Sofiane Benachour}
\author[2]{Humberto Stein Shiromoto}
\author[1]{Vincent Andrieu}
\affil[1]{
\small Universit\'e Lyon 1, Villeurbanne;
CNRS, UMR 5007, LAGEP.
43 bd du 11 novembre, 69100 Villeurbanne, France {\small  https://sites.google.com/site/vincentandrieu/, ms.benachour@gmail.com}}
\affil[2]{\small GIPSA-lab, Grenoble Campus, 11 rue des Math\'ematiques, BP 46,
38402 Saint Martin d'H\`eres Cedex, France {\small  humberto.shiromoto@ieee.org}}

\maketitle

\begin{abstract}
In this paper we consider the problem of global asymptotic stabilization with prescribed local behavior.
We show that this problem can be formulated in terms of control Lyapunov functions.
Moreover, we show that if the local control law has been synthesized employing a LQ approach, then the associated Lyapunov function can be seen as the value function of an optimal problem with some specific local properties.
We illustrate these results on two specific classes of systems: backstepping and feedforward systems.
Finally, we show how this framework can be employed when considering the orbital transfer problem.
\end{abstract}

\section{Introduction}

The synthesis of a stabilizing control law for systems described by nonlinear differential equations has been the subject of great interest by the nonlinear control community during the last three decades.
Depending on the structure of the model, some techniques are now available to synthesize control laws ensuring global and asymptotic stabilization of the equilibrium point.

For instance, we can refer to the popular backstepping approach (see \cite{KrsticKokotovicKanellakopoulos_Book_95,AndrieuPraly_TAC_08} and the reference therein),
or the forwarding approach (see \cite{MazencPraly_TAC_96,JankovicSepulchre_TAC_96,PralyOrtegaKaliora_TAC_01})  and some others based on energy considerations or dissipativity properties   (see \cite{KokotovicArcak_Aut_01} for a survey of the available approaches).

Although the global asymptotic stability of the steady point can be achieved in some specific cases, it remains difficult to address in the same control objective performance issues of a nonlinear system in a closed loop.
However, when the first order approximation of the non-linear model is considered, some performance aspects can be addressed by using linear optimal control techniques (using LQ controller for instance).

Hence, it is interesting to raise the question of synthesizing a nonlinear control law which guarantees the global asymptotic stability of the origin while ensuring a prescribed local linear behavior.
  This problem has been addressed in \cite{PanEzalKrenerKokotovic_TAC_00_LocOpt}.
  In this paper local optimal control laws are designed for systems which admits the existence of a backstepping.

In the present paper we consider this problem in a general manner.
In a first section we will motivate this control problem and we will consider a first strategy based on the design of a uniting control Lyapunov function.
We will show that this is related to an equivalent problem which is the design of a control Lyapunov function with a specific property on the quadratic approximation around the origin.
In a second part of this paper, we will consider the case in which the prescribed local behavior is an optimal LQ controller. In this framework, we investigate what type of performances is achieved by the control solution to the stabilization with prescribed local behavior.
In a third part we consider two specific classes of systems and show how the control with prescribed local behavior can be solved.
With our new context we revisit partially results obtained in \cite{PanEzalKrenerKokotovic_TAC_00_LocOpt}.
Finally in the   fourth   part of the paper, we consider a specific control problem which is the orbital transfer problem.
Employing the Lyapunov approach of Kellet and Praly in \cite{KellettPraly_NOLCOS_04_ContThrustOrbTransfer} we will exhibit a class of costs for which the stabilization with local optimality can be achieved.



\section{Stabilization with prescribed local behavior}
To present the problem under consideration, we introduce a general controlled nonlinear system described by the following ordinary differential equation:
\begin{equation}
\label{eq_system_gen} \dot \XR = \Phi(\XR,u)\ ,
\end{equation}
with the state $\XR$ in $\RR^n$ and $\Phi:\RR^n\times \RR^p\rightarrow\RR^n$ is a $C^1$ function such that $\Phi(0,0)=0$ and $u$ in $\RR^p$ is a control input.
For this system, we can introduce the two matrices $\AR$ in $\RR^{n\times n}$ and $\BR$ in $\RR^{n\times p}$ describing its first order approximation~:
$
\AR:=\frac{\partial \Phi}{\partial \XR}(0,0)\ ,\ \BR:=\frac{\partial \Phi}{\partial u}(0,0)\ .
$
All along the paper hidden in our assumptions, the couple $(\AR,\BR)$ is assumed to be stabilizable.

For system (\ref{eq_system_gen}), the problem we intend to solve can be described as follows:\\[0.5em]
\noindent\textbf{Global asymptotic stabilization with prescribed local behavior:}
Let a linear state feedback law $u=\Ko \XR$ with $\Ko$ in $\RR^{p\times n}$ which stabilizes the first order approximation of system  (\ref{eq_system_gen}) (i.e. $\AR+\BR \Ko$ is Hurwitz) be given.
We are looking for a stabilizing control law $u=\alphao(\XR)$, with $\alphao:\RR^{n}\rightarrow\RR^p$, a locally Lipschitz map differentiable at $0$ such that:
\begin{enumerate}
\item  The origin of the closed-loop system $\dot \XR = \Phi(\XR,\alphao(\XR))$
is globally and asymptotically stable ;

\item The first order approximation of the control law $\alphao$ satisfies the following equality.
\begin{equation}\label{eq_partial_x}
\frac{\partial \alphao}{\partial \XR}(0)=K_o
\ .
\end{equation}
\end{enumerate}

This problem has already been addressed in the literature.
For instance, it is the topic of the papers  \cite{PanEzalKrenerKokotovic_TAC_00_LocOpt,Sahnounetal_IJC_12,BenachourEtAl_CDC_11}.
Note moreover that this subject can be related to the problem of uniting a local and a global control laws as introduced in \cite{TeelKapoor_ECC_97} (see also \cite{Prieur_MCSS_01}).

In this paper, we restrict our attention to the particular case in which the system is input affine.
More precisely we consider systems in the form
\begin{equation}\label{eq_syst_Inputaffine}
\dot \XR = a(\XR) + b(\XR)u\ ,
\end{equation}
with the two $C^1$ functions $a:\RR^n\rightarrow\RR^n$ and $b:\RR^n\rightarrow\RR^{n\times p}$.
In this case we get $\AR = \frac{\partial a}{\partial \XR}(0)$ and $\BR = b(0)$.

Employing the tools developed in \cite{AndrieuPrieur_TAC_10} it is possible to show that merging control Lyapunov function may solve the problem of stabilization with prescribed local behavior.
In the following, we show that working with the control Lyapunov function is indeed equivalent to address this problem.
\begin{theorem}\label{Theo_UnitingLyap}
Given a  linear state feedback law $u=\Ko\XR$ with $\Ko$ in $\RR^{p\times n}$ which stabilizes the first order approximation of system  (\ref{eq_syst_Inputaffine}).
 The following two statements are equivalent.
\begin{enumerate}
\item There exists a locally Lipschitz function $\alphao:\RR^n\rightarrow\RR^p$ solution to the \textit{global asymptotic stabilization with prescribed local behavior problem}.

\item There exists a $C^2$ proper, positive definite function $V:\RR^n\rightarrow\RR_+$ such that the following two properties are satisfied.
\begin{itemize}
\item
If we denote\footnote{In the following, given a $C^2$ function $V:\RR^n\rightarrow\RR$, the notation
$H(V)(\XR)$ is the Hessian matrix in $\RR^{n\times n}$ evaluated at $\XR$ of the function $V$. More precisely, it is the matrix
$
(H(V))_{i,j}(\XR) = \frac{\partial^2 V}{\partial \XR_i\partial \XR_j}(\XR)\ .
$
} $P:=\frac{1}{2}H(V)(0)$, then $P$ is a positive definite matrix.
Moreover this inequality holds.
\begin{equation}\label{eq_Local_AlgebraicLyap}
(\AR+\BR\Ko)'P + P(\AR+\BR\Ko)<0\ ;
\end{equation}

\item Artstein condition is satisfied. More precisely, this implication holds for all $\XR$ in $\RR^n\setminus\{0\}$,
\begin{equation}\label{ArtsteinCondition}
L_bV(\XR)=0\Rightarrow L_aV(\XR)<0,
\end{equation}
  where $L_bV(\cdot)=\partial V/\partial \chi\cdot b(\cdot)$, and $L_aV$ is analogously defined.

\end{itemize}
\end{enumerate}
\end{theorem}
\if\longversion1
\begin{proof}
\noindent $1)\Rightarrow 2)$
The proof of this part of the theorem is based on recent results obtained in \cite{AndrieuPrieur_TAC_10}.
Indeed, the design of the function $V$ is obtained from the uniting of a quadratic local control Lyapunov function (denoted $V_0$) and a global control Lyapunov function (denoted $V_\infty$) obtained employing a converse Lyapunov theorem.

First of all, employing the converse Lyapunov theorem  of Kurzweil in \cite{Kurzweil_AMST_56}, there exists a $C^\infty$ function $V_\infty:\RR^n\rightarrow\RR_+$ such that
$
\frac{\partial V_\infty}{\partial \XR}(\XR)[a(\XR)+b(\XR)\alpha_o(\XR)]<0\ ,\ \forall \ \XR\neq 0\ .
$
On the other hand, $\AR+\BR \Ko$ being Hurwitz, there exists a matrix $P$ such that the algebraic Lyapunov inequality (\ref{eq_Local_AlgebraicLyap}) is satisfied.
Let $V_0$ be the quadratic function $V_0(\XR) = \XR'P \XR$.
Due to the fact that $\Ko$ satisfies equation (\ref{eq_partial_x}) it yields that the matrix $\AR+\BR \Ko$ is the first order approximation of the system  (\ref{eq_syst_Inputaffine}) with the control law $u=\alphao(x)$.
 Consequently, it implies that there exists a positive real number $\epsilon_1$ such that
$
\frac{\partial V_0}{\partial \XR}(\XR)[a(\XR)+b(\XR)\alpha_o(\XR)]<0\ ,\ \forall \ 0<|\XR|\leq \epsilon_1\ .
$
This implies that the time derivative of the two control Lyapunov functions $V_0$ and $V_\infty$ can be made negative definite with the same control law in a neighborhood of the origin.
Employing \cite[Theorem 2.1]{AndrieuPrieur_TAC_10}, it yields the existence of a function $V:\RR^n\rightarrow\RR_+$ which is  $C^2$ at the origin and a positive real number $\epsilon_2$ such that the following two properties hold.
\begin{itemize}
\item For all $\XR$ in $\RR^n\setminus\{0\}$,
$
\frac{\partial V}{\partial \XR}(\XR)[a(\XR)+b(\XR)\alpha_o(\XR)] <0\ .
$
Hence,   Equation   (\ref{ArtsteinCondition}) is satisfied\ ;
\item For all $\XR$ in $\RR^n$ such that $|\XR|\leq\epsilon_2$, we have
$
V(\XR) = V_0(\XR)\ .
$
Consequently $\HR(V)(0) = 2P$.
\end{itemize}
\noindent $2)\Rightarrow 1)$
Let $Q$ be the positive definite matrix defined as,
$
Q:=-(\AR+\BR K_o)'P + P(\AR+\BR K_o)\ .
$
Employing the local approximation of the Lyapunov function $V$, it is possible to find
$r_0$ such that
$$
L_aV(\XR) + L_bV(\XR)\Ko \XR<0\ ,\ \forall \XR\in\{0<V(\XR)\leq r_0\}\ .
$$
This implies that the control Lyapunov function $V$ satisfies the small control property (see \cite{Sontag_SCL_89}).
Hence,  we get the existence of a control law $\alpha_\infty$ (given by Sontag's universal formulae introduced in \cite{Sontag_SCL_89}) such that this one satisfies for all $\XR\neq 0$
$$
L_aV(\XR) + L_bV(\XR)\alpha_{  \infty  }(\XR)<0\ .
$$
A solution to the stabilization with prescribed local problem can be given by the control law
$
\alphao(\XR) = \rho(V(\XR))\alpha_\infty(\XR) + (1-\rho(V(\XR)))\Ko \XR
$
where $\rho:\RR_+\rightarrow [0,1]$ is any locally Lipschitz function such that
$
\rho(s) = \left\{
\begin{array}{cc}
0\ ,\ &s\leq \frac{r_0}{2}\ ,\\
1\ ,\ &s\geq r_0\ .\\
\end{array}
\right.
$
Note that with this selection, it yields that equality (\ref{eq_partial_x}) holds.
Moreover, we have along the solution of the system (\ref{eq_syst_Inputaffine})
\\[0.5em]$\displaystyle
\left.\dot V(\XR)\right|_{u=\alphao(\XR)}=
\rho(V(\XR))\left.\dot V(\XR)\right|_{u=\alpha_\infty}\displaystyle
+ (1-\rho(V(\XR)))\left.\dot V(\XR)\right|_{u=\Ko \XR}<0
$\\[0.5em]
Hence, we get the result.
\end{proof}
\else
This proof is based on the uniting of control Lyapunov function as developed in \cite{AndrieuPrieur_TAC_10}. It can be found in the long version of this paper in \cite{BenachourAndrieu_HAL_13}.
\fi

From this theorem, we see that looking for a global control Lyapunov function locally assigned by the prescribed local behavior and looking for the controller itself are equivalent problems.

\section{Locally optimal and globally inverse optimal control laws}

If one wants to guarantee a specific behavior on the closed loop system, one might want to find a control law which minimizes a specific cost function.
More precisely, we may look for a stabilizing control law which minimizes the criterium
\begin{equation}\label{eq_cost}
J(\XR;u) = \displaystyle\int_0^{+\infty}
q(X(\XR,t;u))+u(t)'r(X(\XR,t;u))u(t)
dt,
\end{equation}
where $X(\XR,t;u)$ is the solution of the system (\ref{eq_syst_Inputaffine}) initiated from $\XR_0=\XR$ at $t=0$ and employing the control $u:\RR_+\rightarrow\RR^p$, $q:\RR^n\rightarrow\RR_+$ is a continuous function and $r$ is a continuous function which values $r(\XR)$ are symmetric positive definite matrices.

The control law which solves this minimization problem (see \cite{SepulchreJankovic_Book_97}) is given as a state feedback
\begin{equation}\label{eq_OptCont}
u = -\frac{1}{2}r(\XR)^{-1}L_bV(\XR)'\ ,
\end{equation}
where $V:\RR^n\rightarrow\RR_+$ is the solution with $V(0)=0$ to the following Hamilton-Jacobi-Bellman equation for all $\XR$ in $\RR^n$
\begin{equation}\label{eq_HJB}
q(\XR) + L_aV(\XR) - \frac{1}{4}L_bV(\XR)r(\XR)^{-1}L_bV(\XR)'=0\ .
\end{equation}

Given a function $q$ and a function $r$, it is in general difficult or impossible to solve the so called HJB equation.
However, for linear system, this might be solved easily.
If we consider the first order approximation of the system (\ref{eq_syst_Inputaffine}), and
given a positive definite matrix $R$ and a positive semi definite matrix $Q$ we can introduce the quadratic cost:
\begin{equation}\label{eq_cost_quadra}
J(\XR;u) = \displaystyle\int_0^{+\infty}
\left[
X(\XR,t;u)'Q X(\XR,t;u) + u(t)'Ru(t)
\right]dt,
\end{equation}

In this context, solving the HJB equation can be rephrased in solving the   algebraic Riccati   equation given as
\begin{equation}\label{eq_HJB_quadra}
P\AR+\AR'P-P\BR R^{-1}\BR'P+Q=0\ .
\end{equation}

It is well known that provided, the couple $(\AR,\BR)$ is controllable, it is possible to find a solution to this equation.
Hence, for the first order approximation, it is possible to solve the optimal control problem when considering a cost in the form of (\ref{eq_cost_quadra}).

From this discussion, we see that an interesting control strategy is to solve the stabilization with prescribed local behavior with the local behavior obtained solving LQ control strategy.
Note however that once we have solved this problem, one may wonder what type of performances has been achieved by this new control law.
The following Theorem addresses this point and is inspired from \cite{SepulchreJankovic_Book_97} (see also \cite{Praly_Poly_08}).
Following Theorem  \ref{Theo_UnitingLyap}, this one is given in terms of control Lyapunov functions.
\begin{theorem}[Local optimality and global inverse optimality]\label{Theo_LocOptStab}
Given two positive definite matrices $R$ and $Q$.
Assume there exists a $C^2$ proper positive definite function $V:\RR^n\rightarrow\RR_+$
such that the following two properties hold.
\begin{itemize}
\item
The matrix $P:=H(V)(0)$ is positive definite matrix and satisfies the following equality.
\begin{equation}
P\AR+\AR'P-P\BR R^{-1}\BR'P+Q=0\ ;
\end{equation}

\item   Equation \eqref{ArtsteinCondition}   is satisfied.
\end{itemize}
Then there exist
$q:\RR^n\rightarrow\RR_+$ a continuous function, $C^2$ at zero and $r$ a continuous function   whose   values $r(\XR)$ are symmetric positive definite matrices
such that the following properties are satisfied.
\begin{itemize}
\item The function $q$ and $r$ satisfy
\begin{equation}\label{eq_LocOpt}
H(q)(0)=2Q\ ,\ r(0)=R\ ;
\end{equation}
\item The function $V$ is a value function associated to the cost (\ref{eq_cost}). More precisely, $V$ satisfies the HJB equation (\ref{eq_HJB}).
\end{itemize}
\end{theorem}
\if\longversion1
\begin{proof}
This proof is inspired from some of the results of \cite{Praly_Poly_08}.

First of all, there exists a positive real number $r_0$ such that for all $\XR$ such that $0<V(\XR)\leq r_0$ we have
$
-L_fV(\XR) +\frac{1}{4} L_gV(\XR)R^{-1}L_gV(\XR)'>0\ .
$
Now, for all $k$ in $\NN$, we consider $C_k$ the subset of $\RR^n$ defined as
$
C_k = \{\XR, kr_0\leq V(\XR)\leq (k+1)r_0\}\ .
$
Note that since the function $V$ is proper, for all $k$ the set $C_k$ is a compact subset.
Assume for the time being that for all $k$ there exists $\ell_k$ in $\RR_+$ such that~:
\begin{equation}\label{eq_CLF_k}
L_aV(\XR) -\frac{\ell_k}{4} L_bV(\XR)R^{-1}L_gV(\XR)'<0\ ,\ \forall \XR\in C_k\ .
\end{equation}
Let $\mu$ be any continuous function such that,
$$
\mu(s)
\left\{
\begin{array}{ll}
=1\ ,\ &s\leq \frac{r_0}{2}\ ,\\
\geq 1\ ,\ &\frac{r_0}{2}\leq s\leq r_0\ ,\\
\geq \ell_k\ ,\ &kr_0\leq s\leq (k+1)r_0\ .
\end{array}
\right.
$$
Moreover, let
$
r(\XR):= \frac{1}{\mu(V(\XR))}R\ ,
$
and
$
q(\XR):=-L_aV(\XR) +\frac{1}{4} L_bV(\XR)r(\XR)^{-1}L_bV(\XR)'\ .
$
With (\ref{eq_CLF_k}) and the definition of $\mu$, it yields,
$
q(\XR)>0\ ,\ \forall \XR\neq 0\ .
$
Hence, $V$ is solution to the associated HJB equation.
Note moreover that we have $r(0)=R$ and
$
\frac{1}{2}H(q)(0)=\AR'P+P\AR - P\BR R^{-1}\BR 'P = Q\ .
$
Hence, the result.

In conclusion, to get the result, we only need to show that for all $k$ in $\NN$, there exists $\ell_k$ such that (\ref{eq_CLF_k}) is satisfied.
Assume this is not the case for a specific $k$ in $\NN$.
This implies that for all $j$ in $\NN$ there exists $x_j$ in $C_k$ such that
$
L_aV(\XR_j) -\frac{j}{4} L_bV(\XR_j)R^{-1}L_bV(\XR_j)'\geq 0\ .
$
The sequence $x_j$ being in a compact set, we know there exists a converging subsequence denoted $\left(\XR_{j_\ell}\right)_{\ell\in\NN}$ which converges toward a cluster point denoted $\XR^*$ in $C_k$.
The previous inequality can be rewritten as:
$
\frac{L_aV(\XR_{j_\ell})}{j_\ell} \geq \frac{1}{4} L_bV(\XR_{j_\ell})R^{-1}L_bV(\XR_{j_\ell})'\geq 0\ .
$
Letting $j_\ell$ goes to infinity yields
$
L_aV(\XR^*)\geq 0$ and $L_bV(\XR^*)=0$.
With \eqref{ArtsteinCondition}, this implies that $L_aV(\XR^*)<0$ hence a contradiction.
This ends the proof.
\end{proof}
\else
This proof is inspired from some of the results of \cite{Praly_Poly_08} and can be found in the long version of this paper \cite{BenachourAndrieu_HAL_13}.
\fi

This Theorem establishes that if we solve the stabilization with a prescribed local behavior, we may design a control law $u=\alphao(\XR)$ such that this one is solution to an optimal control problem and such that the local approximation of the associated cost is exactly the one of the local system.
This framework has already been studied in the literature in \cite{PanEzalKrenerKokotovic_TAC_00_LocOpt}. In this paper is addressed the design of a backstepping with a prescribed local optimal control law.
In our context we get a Lyapunov sufficient condition to design a globally and asymptotically stabilizing optimal control law with prescribed local cost function.



\section{Some sufficient conditions}
In this section we give some sufficient conditions allowing us to solve the stabilization with prescribed local behavior problem.
The first result is obtained from the tools developed in \cite{AndrieuPrieur_TAC_10}.
It assumes the existence of a global control Lyapunov function and a sufficient condition is given in terms of a matrix inequality.
In the second and third results we give some structural conditions on the vector field to avoid a matrix inequality.
%

\subsection{Based on matrix inequalities }
\label{SecLMI}

The first solution to solve the stabilization with prescribed local behavior is to follow the result of \cite{AndrieuPrieur_TAC_10} and to assume that there exists a global control Lyapunov function which can be modified locally in order to fit in the context of Theorem \ref{Theo_UnitingLyap}.

\begin{assumption}\label{ass_GlobalCLF}
There exists a positive definite and $C^2$ function $V_\infty:\RR^n\rightarrow\RR_+$ such that the following holds.
\begin{enumerate}
\item The implication (\ref{ArtsteinCondition}) is satisfied.
\item The function $V_\infty$ is locally quadratic. i.e. $P_\infty = H(V)(0)$ is a positive definite matrix.
\end{enumerate}
\end{assumption}

In this context the result obtained from \cite{AndrieuPrieur_TAC_10} may be formalized as follows.
\begin{theorem}(\cite{AndrieuPrieur_TAC_10})
Let Assumption \ref{ass_GlobalCLF} be satisfied.
Let $\Ko$ in $\RR^{p\times n}$ be a matrix such that $\AR + \BR\Ko$
is Hurwitz with $\AR$ and $\BR$ defined in (\ref{eq_FirstOrdeBackstepping}).
If there exists $K_u$ in $\RR^{p\times n}$ and a positive definite matrix $P$ in $\RR^{n\times n}$  such that these matrix inequalities are satisfied
\begin{equation}\label{MatrixEq}
\begin{array}{rcc}
(\AR + \BR \Ko)'P + P(\AR + \BR \Ko)&<&0\ ,\\
(\AR + \BR K_u)'P + P(\AR + \BR K_u)&<&0\ ,\\
(\AR + \BR K_u)'P_\infty + P_\infty(\AR + \BR K_u)&<&0\ ,
\end{array}
\end{equation}
then there exists a smooth function $\alphao:\RR^n\rightarrow\RR^p$ which solves the global asymptotic stabilization with prescribed local  behavior.
\end{theorem}
\begin{proof}
The proof of this result is a direct consequence of the tools developed in \cite{AndrieuPrieur_TAC_10}.
\end{proof}

In inequalities (\ref{MatrixEq}), $P$ and $K_u$ are the unknown. This implies that this inequality is not linear.
However by introducing some new variables, it is possible to give a (conservative) linear relaxation which allows the use of the tools devoted to solve linear matrix inequalities (see \cite{AndrieuPrieurTarbouriechArzelier_SCL_10} for instance).

\subsection{Strict feedback form}

Following the work of \cite{PanEzalKrenerKokotovic_TAC_00_LocOpt}, consider the case in which system (\ref{eq_syst_Inputaffine}) with state $\XR=(y,x)$ can be written in the following form
\begin{equation}\label{eq_BackSyst}
\dot y = h_1(y)+h_2(y)x\ ,\ \dot x= f(y,x) + g(y,x)u\ .
\end{equation}
with $y$ in $\RR^{n_y}$, $x$ in $\RR$ and $g(y,x)\neq 0$ for all $(y,x)$.

In this case, the first order approximation of the system is
\begin{equation}\label{eq_FirstOrdeBackstepping}
\AR = \left[
\begin{array}{cc}
H_1 & H_2\\
F_1 & F_2
\end{array}
\right]\ ,\ \BR  = \left[
\begin{array}{c}
0\\
G
\end{array}\right]\ ,
\end{equation}
with $H_1 = \frac{\partial h_1}{\partial y}(0)$, $H_2 = h_2(0)$, $F_1 = \frac{\partial f}{\partial y}(0,0)$, $F_2 = \frac{\partial f}{\partial x}(0,0)$, $G=g(0,0)$.

For this class of system we make the following assumption.
\begin{assumption}\label{ass_Backstepping}
For all   couples   $(K_y,P_y)$ with $K_y$ in $\RR^{n_y}$ and $P_y$ a positive definite matrix in $\RR^{n_y \times n_y}$ such that
$
P_y(H_1+H_2K_y) + (H_1+H_2K_y)'P_y<0\ ,
$
there exists a smooth function $V_y:\RR^{n_y}\rightarrow\RR_+$ such that
$H(V_y)(0)=2P_y$ and such that for all $y\neq 0$
\begin{equation}\label{ArtsteinConditiony}
L_{h_2}V_y(y)=0\Rightarrow L_{h_1}V_y(y)<0\ .
\end{equation}
\end{assumption}

With Theorem \ref{Theo_UnitingLyap}, this assumption establishes that  the stabilization with prescribed local behavior is satisfied for the $y$ subsystem seeing $x$ as the control input.

For this class of system, we have the following theorem
which can already be found in \cite{PanEzalKrenerKokotovic_TAC_00_LocOpt} when restricted to locally optimal controllers.
\begin{theorem}[Backstepping Case]\label{Theo_BackStepping}
Let Assumption \ref{ass_Backstepping} be satisfied.
Let $\Ko$ in $\RR^{p\times n}$ be a matrix such that $\AR + \BR\Ko$
is Hurwitz with $\AR$ and $\BR$ defined in (\ref{eq_FirstOrdeBackstepping}).
Then there exists a smooth function $\alphao:\RR^n\rightarrow\RR^p$ which solves the global asymptotic stabilization with prescribed local  behavior.
\end{theorem}

\if\longversion1

\begin{proof}
Let $P$ be a positive definite matrix such that the algebraic Lyapunov inequality (\ref{eq_Local_AlgebraicLyap}) is satisfied.
This matrix can be rewritten $P=\left[\begin{array}{cc}
P_{11} & P_{12}\\P_{12}' & P_{22}\end{array}\right]$
with $P_{22}, P_{12}, P_{22}$  matrices respectively in $\RR^{n_y\times n_y}, \RR^{n_y\times n}, \RR$.
Let $T$ be the matrix in $\RR^{(n_y+1)\times n_y}$ defined as\footnote{Given a positive integer $n$, the notation $\Id_n$ is the identity matrix in $\RR^{n\times n}$.}
$
T = \left[
\begin{array}{c}
\Id_{n_y}\\[0.1em]-\frac{P_{12}'}{P_{22}}
\end{array} \right]\ .
$
Note that this matrix satisfies
$
T' P =
\left[\begin{array}{cc}P_y & 0\end{array}\right]\ ,\ T'P\BR = 0\ ,
$
where $P_y = P_{11}-P_{12}P_{22}^{-1}P_{12}'$ is the Schur complement of $P$.

By pre and post multiplying inequality (\ref{eq_Local_AlgebraicLyap}) respectively by $T'$ and $T$ it yields the following inequality.
\begin{equation}\label{eq_Lyap_SubSystem}
P_y
\left(H_1-H_2 \frac{P_{12}'}{P_{22}}\right) +
\left(H_1-H_2 \frac{P_{12}'}{P_{22}}\right)'P_y<0\ .
\end{equation}
The matrix $P$ being positive definite, its Schur complement $P_y$ is also positive definite.
Hence, inequality (\ref{eq_Lyap_SubSystem}) can be seen as a Lyapunov inequality and $x=-\frac{P_{12}}{P_{22}}y$ as a stabilizing local controller for the $y$ subsystem with $P_y$ as associated Lyapunov matrix.
With Assumption \ref{ass_Backstepping}, and Theorem \ref{Theo_UnitingLyap} we know there exist a smooth function $\alpha_y:\RR^{n_y}\rightarrow\RR$ and a smooth function $V_y:\RR^{n_y}\rightarrow\RR_+$ such that the following two properties hold.
\begin{itemize}
\item The origin of the system $\dot y = h_1(y)+h_2(y)\alpha_y(y)$ is globally and asymptotically stable with associated Lyapunov function $V_y$. More precisely, we have
\begin{equation}\label{Backstepping1}
\frac{\partial V_y}{\partial y}(y)\left[h_1(y)+h_2(y)\alpha_y(y)\right]<0\ ,\ \forall y\neq 0\ ;
\end{equation}
\item We have the local properties
$
\frac{\partial \alpha_y}{\partial y}(0) = -\frac{P_{12}}{P_{22}}\ ,\ H(V_y)(0) = 2P_y\ .
$
\end{itemize}
Consider now the function
\begin{equation}
V(\XR) =  V_y(y)+P_{22}(x-\alpha_y(y))^2\ .
\end{equation}
Note that this function is proper and positive definite.
Moreover, we have
$
L_bV(\XR)=2P_{22}(x-\alpha_y(y))g(x,y)\ .
$
Moreover, since it is assumed that $g(x,y)\neq 0$, this implies
$ L_bV(\XR)=0, \XR\neq 0 \Rightarrow x=\alpha_y(y)\ .$
Note that when $x=\alpha_y(y)$, with (\ref{Backstepping1}) we have for all $y\neq 0$
$
L_aV(\XR)= \frac{\partial V_y}{\partial y}(y)h(y,\alpha_y(y))<0\ .
$
Hence,   Equation \eqref{ArtsteinCondition}   is satisfied.
Finally, we have the following equality.
$
H(V)(0) = 2P\ .
$
Hence, with Theorems \ref{Theo_UnitingLyap} we get the result.
\end{proof}

\fi

Note that with Theorem \ref{Theo_LocOptStab}, this theorem establishes that given $Q$, a positive definite matrix in $\RR^{n_y\times n_y}$, and $R$,  a positive real number, then there exist $q$, $r$ and   $\alphao$ which is solution to an optimal control problem with cost $J(\XR,u)$ defined in (\ref{eq_cost}),
with $q$ and $r$ which satisfy (\ref{eq_LocOpt}).
In other words we can
design a globally and asymptotically stabilizing optimal control law with prescribed local cost function as already seen in \cite{PanEzalKrenerKokotovic_TAC_00_LocOpt}.

\subsection{Feedforward form}

Following our previous work in \cite{BenachourEtAl_CDC_11}, consider the case in which the system with state $\XR=(y,x)$ can be written in the form
\begin{equation}\label{eq_ForwardingSyst}
\dot y = h(x)\ ,\ \dot x= f(x) + g(x)u\ ,
\end{equation}
with $y$ in $\RR$, $x$ in $\RR^{n_x}$.
Note that to oppose to what has been done in the previous subsection, now the state component $y$ is a scalar and $x$ is a vector.
Note moreover that the functions $h$, $f$ and $g$ do not depend of $y$.
This restriction on $h$ has been partially removed in our recent work in \cite{BenachourEtAl_TAC_13}.

The first order approximation of the system is denoted by
\begin{equation}\label{eq_FirstOrdeForwarding}
\AR = \left[
\begin{array}{cc}
0 & H\\
0 & F
\end{array}
\right]\ ,\ \BR  = \left[
\begin{array}{c}
0\\
G
\end{array}\right]\ ,
\end{equation}
with $H = \frac{\partial h}{\partial x}(0)$, $F = \frac{\partial f}{\partial x}(0)$, $G=g(0)$.

For this class of system we make the following assumption.
\begin{assumption}\label{ass_Forwarding}
For all   couples   $(K_x,P_x)$ with $K_x$ in $\RR^{p\times n_x}$ and $P_x$ a positive definite matrix in $\RR^{n_x \times n_x}$ such that
$
P_x(F+GK_x) + (F+GK_x)'P_x<0\ ,
$
there exists a smooth function $V_x:\RR^{n_x}\rightarrow\RR_+$ such that
$H(V_x)(0)=2P_x$ and such that for all $x\neq 0$
\begin{equation}\label{ArtsteinConditionx}
L_{g}V_x(x)=0\Rightarrow L_{f}V_x(x)<0\ .
\end{equation}
%
%
\end{assumption}

This assumption establishes that the stabilization with prescribed local behavior is satisfied for the $x$ subsystem.
With this Assumption we have the following theorem   whose   proof can be found in \cite{BenachourEtAl_CDC_11}.
\begin{theorem}[Forwarding Case]\label{Theo_Forwarding}
Let Assumption \ref{ass_Forwarding} be satisfied.
Let $\Ko$ in $\RR^{p\times n}$ be a vector such that  the matrix
$\AR + \BR\Ko$ is Hurwitz with $\AR$ and $\BR$ defined in (\ref{eq_FirstOrdeForwarding}).
Then there exists a smooth function $\alphao:\RR^n\rightarrow\RR^p$ which solves the global asymptotic stabilization with prescribed local  behavior.
\end{theorem}

Similarly to the backstepping case  this theorem with Theorem \ref{Theo_LocOptStab} establish that given $Q$, a positive definite matrix in $\RR^{n\times n}$, and $R$,  a positive real number, there exists $q$, $r$ and  $\alphao$ which is solution to an optimal control problem with cost $J(\XR,u)$ defined in (\ref{eq_cost}),
with $q$ and $r$ which satisfy (\ref{eq_LocOpt}).
Consequently, similarly to the backstepping case, we can
design a globally and asymptotically stabilizing optimal control law with prescribed local cost function.

\section{Illustration on the orbital transfer problem}
As an illustration of the results described in the previous sections, we consider the problem of designing a control law which ensures the orbital transfer of a satellite from one orbit to another.
In this section we consider the approach developed in \cite{KellettPraly_NOLCOS_04_ContThrustOrbTransfer} where a bounded stabilizing control law was developed.
More precisely, we study the class of optimal control law (in the LQ sense) that can be synthesized.
This may be of interest since, as mentioned in \cite{Bombrun_PhD_07transferts}, it is difficult to consider performance issues with this control law.

  Consider the example presented in \cite{KellettPraly_NOLCOS_04_ContThrustOrbTransfer}. Applying a suitable coordinate change
it yields
\begin{equation}\label{eq_SystemTransfOrbit}
\left\{
\begin{array}{rcl}
\dot { \XR}_1&=& \overline{\eta} \sqrt{\XR_4}(1+\XR_2)^2
- \eta
- \frac{\nu}{\sqrt{p_0^3}} \frac{\XR_6\sqrt{\XR_4^3}}{1+ \XR_2}u_h\\[0.5em]
\dot { \XR}_2&=&- \eta (1+ \XR_2)^2 \XR_3
\\[0.5em]
\dot { \XR}_3&=& \eta [(1+ \XR_2)^2 \left[\frac{\XR_4}{p_0}(1+\XR_2)-1\right]+ \nu  u_r
\\[0.5em]
\dot \XR_4 &=& 2 \frac{\nu}{\sqrt{p_0^3}} \frac{\sqrt{\XR_4^5}}{1+ \XR_2}u_\theta\\[0.5em]
\dot \XR_5&=& \overline{\eta} \sqrt{\XR_4}(1+\XR_2)^2\XR_6+
 \overline{\nu} \frac{1+\XR_5^2-\XR_6^2}{2\sqrt{\XR_4}(1+ \XR_2)} u_h  \\
\dot \XR_6 &=& - \overline{\eta} \sqrt{\XR_4}(1+\XR_2)^2\XR_5+ \overline{\nu} \frac{\XR_5\XR_6}{\sqrt{\XR_4}(1+ \XR_2)} u_h,
\end{array}
\right.
\end{equation}
  where $p_0$,
\begin{equation*}
    \begin{array}{rclcrcl}
        \nu&=&\sqrt{\frac{p_0}{\mu}},&&\eta&=&\frac{1}{p_0\nu},\\
        \overline{\nu}&=&\nu\sqrt{p_0},&&\overline{\eta}&=&\frac{\eta}{\sqrt{p_0}},
    \end{array}
\end{equation*}
are constants values. Concerning the states, in this new coordinate system, $\chi_1$ is the true longitude, $\chi_2$ and $\chi_3$ are the $x$ and $y$ components of the eccentricity vector, $\chi_4$ is the parameter, $\chi_5$ and $\chi_6$ are the $x$ and $y$ components of the momentum vector.

In compact form, the previous system is simply:
$
\dot \XR = a(\XR) + b_r(\XR)u_r + b_\theta(\XR)u_\theta+ b_h(\XR)u_h\ .
$
The first order approximation of this system around the equilibrium is given as
$
\AR =  \eta
\left[
\begin{array}{cccccc}
0 & 2&0&\frac{1}{2}&0&0\\
0 &0& 1&0&0&0\\
0 & 1&0&\frac{1}{p_0}&0&0\\
0&0&0&0&0&0\\
0 & 0&0&0&0&1\\
0 & 0&0&0&-1&0
\end{array}
\right]\
$
and
$
\BR =  \nu
\left[
\begin{array}{ccc}
0&0&0\\
0&0 &0\\
1&0 &0\\
0&2p_0&0\\
0&0&\frac{1}{2} \\
0&0&0
\end{array}
\right].
$
Note that these matrices can be rewritten as
$
\AR =\texttt{diag}\{\tilde \AR, A_1\}\ ,\
\tilde \AR =\left[
\begin{array}{cc}
A_0 & A_2 \\
0_{13} &0
\end{array}
\right]
$
and
$
\BR =\texttt{diag}\{\tilde \BR, B_2\}\ ,\ \tilde\BR =
\left[
\begin{array}{cc}
B_0 & 0_{31}\\
0 &  \frac{2}{\eta}
\end{array}
\right]
$
where
$
A_0 =   \eta
\left[
\begin{array}{ccc}
  0   & 2&0\\
0 &0& 1\\
0&1&0
\end{array}
\right]\ ,\
A_1 = \eta
\left[
\begin{array}{cc}
0& 1 \\
- 1  &0
\end{array}
\right]\ ,
$
and,
$
A_2 =  \eta
\left[
\begin{array}{c}
\frac{1}{2} \\
0\\
\frac{1}{p_0}
\end{array}
\right]\ ,
B_0 =  \nu
\left[
\begin{array}{c}
0\\
0\\
1
\end{array}
\right]
\ ,
\
B_2 =  \nu
\left[
\begin{array}{c}
\frac{1}{2}\\
0
\end{array}
\right]\ .
$

The control strategy developed in \cite{KellettPraly_NOLCOS_04_ContThrustOrbTransfer} was to successively apply backstepping, forwarding and dissipativity properties.

With the tools developed in the previous sections, we are able to solve the locally optimal control problem for a specific class of quadratic costs as described by the following theorem.
\begin{theorem}[Locally optimal stabilizing control law]
Given $Q_0$ a positive definite matrix in $\RR^{3\times 3}$ and $R_0$ in $\RR_+$.
Let $P_0$ be the solution of the (partial) algebraic   Riccati   equation:
\begin{equation}\label{eq_PartialAlgRiccati}
A_0P_0 + P_0A_0 - P_0B_0R_0^{-1}B_0'P_0 = -Q_0\ .
\end{equation}
Then for all  positive real numbers $R_0,R_1,R_2,\rho_1, \rho_2$ such that the matrix
$
Q = \texttt{diag}\{\tilde Q,\rho_2^2 B_2R_2^{-1}B_2'\}\ ,\ \tilde Q = \left[\begin{array}{cc}
Q_0 & P_0A_2 \\
A_2'P_0 &  \frac{4}{\eta^2} \rho_1^2 R_1^{-1}
\end{array}\right]
$
is positive,
there exists $q$ and $r$ and a globally asymptotically stabilizing control law  $(u_r,u_\theta,u_h) = \alphao(\XR)$ which is solution to an optimal control problem with cost $J(\XR;u)$ defined in (\ref{eq_cost}),
with $q$ and $r$ which satisfy (\ref{eq_LocOpt}).
\end{theorem}
\begin{proof}
First of all, when $u_\theta=u_h=0$ and when $\XR_4=p_0$, then the dynamics of the $(\XR_1, \XR_2,\XR_3)$ subsystem satisfies
\begin{equation}\label{eq_PartialSystem}
\left\{
\begin{array}{rcl}
\dot {\XR}_1&=& \eta \left[(1+\XR_2)^2-1\right]\\
\dot {\XR}_2&=&- \eta (1+ \XR_2)^2\XR_3\\
\dot {\XR}_3&=& \eta (1+ \XR_2)^2\XR_2+ \nu  u_r.
\end{array}
\right.
\end{equation}

It can be noticed setting $y:=\XR_3$ and $x:=\XR_2$ the $(\XR_2,\XR_3)$ subsystem is in the strict feedback form (\ref{eq_BackSyst}). Note that employing Theorem \ref{Theo_BackStepping}, it yields that for this system all locally stabilizing linear behaviors can be achieved.

Moreover, setting $y:=\XR_1$ and $x:=(\XR_2,\XR_3)$ the $(\XR_1,\XR_2,\XR_3)$ subsystem is in the feedforward form (\ref{eq_ForwardingSyst}). Note that employing Theorem \ref{Theo_Forwarding}, it yields that for this system all locally stabilizing linear behaviors can be achieved.

Hence, with Theorem \ref{Theo_UnitingLyap}, it yields that given $P_0$ which by (\ref{eq_PartialAlgRiccati}) is a CLF for the first order approximation of the system (\ref{eq_PartialSystem}) there exists a smooth function $V_0:\RR^3\rightarrow\RR_+$ such that
\begin{itemize}
\item $V_0$ is a CLF for the  $(\XR_1,\XR_2,\XR_3)$ subsystem when considering the control $u_r$ and when $\XR_4=p_0$, i.e. for the system (\ref{eq_PartialSystem})\ ;
\item $V_0$ is locally quadratic and satisfies $H(V_0)(0) = 2P_0$\ .
\end{itemize}

Let $\tilde V:\RR^4\rightarrow\RR_+$ be the function defined by
$
\tilde V(\XR_1,\XR_2,\XR_3,\XR_4) = V_0(\XR_1,\XR_2,\XR_3) +   V_1(\XR_4)\ ,
$
with $V_1(\XR_4)= \rho_1(p_0-\XR_4)^2$.
Note that this function is such that
$
H(\tilde V)(0,0,0,p_0) = 2\tilde P\ ,\ \tilde P = \texttt{diag}\left\{P_0,\rho_1\right\}\ .
$
Employing (\ref{eq_PartialAlgRiccati}), it can be checked that $\tilde P$ satisfies the (partial) algebraic   Ricatti
$
\tilde P\tilde \AR+\tilde \AR'\tilde P-\tilde P\tilde\BR\tilde R^{-1}\tilde \BR'\tilde P + \tilde Q=0\ ,
$
with $\tilde R=\texttt{diag}\{R_1,R_2\}$.
We will show that this function is also a control Lyapunov function when considering the $(\XR_1,\XR_2,\XR_3,\XR_4)$ subsystem in (\ref{eq_SystemTransfOrbit}) with the control inputs $u_r$ and $u_\theta$.
Consider the set of point in $\RR^4$ such that $L_{b_r}\tilde V(\XR) = L_{b_\theta}\tilde V(\XR) =0$.
Note that $L_{b_\theta}\tilde V(\XR) =0$ implies that $\XR_4 = p_0$.
With the CLF property for the system (\ref{eq_PartialSystem}), it yields that in this set $L_aV_0(\XR) <0 $ for all $(\XR_1,\XR_2,\XR_3)\neq 0$.
Consequently, $L_a(\tilde V)(\XR) <0 $ for all $(\XR_1,\XR_2,\XR_3, \XR_4-p_0)\neq 0$ such that $L_{b_r}\tilde V(\XR) = L_{b_\theta}\tilde V(\XR) =0$.
Hence with Theorem \ref{Theo_LocOptStab} we get the existence of $\tilde q:\RR^4\rightarrow\RR_+$ a continuous function, $C^2$ at zero and $\tilde r$ a continuous function which values $r(\XR)$ are symmetric positive definite matrices
such that:
\begin{itemize}
\item The function $\tilde q$ and $\tilde r$ satisfy the following property
\begin{equation}\label{eq_LocOptPartial}
H(\tilde q)(0,0,0,p_0)=2\tilde Q\ ,\ r(0,0,0,p_0)=\tilde R\ .
\end{equation}
\item The function $\tilde V$ is a value function associated to the cost (\ref{eq_cost}) with $\tilde q$ and $\tilde r$. More precisely, $\tilde V$ satisfies the HJB equation (\ref{eq_HJB}) when considering the  $(\XR_1,\XR_2,\XR_3,\XR_4)$ subsystem in (\ref{eq_SystemTransfOrbit}).
\end{itemize}

Finally, let $V:\RR^6\rightarrow\RR_+$ be defined by
$
V(\XR) = \tilde V(\XR_1,\XR_2,\XR_3,\XR_4) + V_2(\XR_5, \XR_6)\ ,
$
with $V_2(\XR_5, \XR_6)=\rho_2 (\XR_5^2 + \XR_6^2)$.
Moreover, consider $q$ the positive semi definite function $q$ defined as
$
q(\XR) = \tilde q(\XR_1,\XR_2,\XR_3,\XR_4) + \frac{1}{4}(L_{b_r}V(\XR))^2R_2^{-1}\ ,
$
and $r$ defined as
$
r(\XR) = \texttt{diag}\{\tilde r(\XR),R_2\}\ .
$
Note that the following properties are satisfied.
\begin{itemize}
\item The function $q$ and $r$ satisfy
\begin{equation}
H(\tilde q)(0)=2Q\ ,\ r(0)=\texttt{diag}\{R_1,R_2,R_3\}\ ;
\end{equation}
\item The function $ V$ is a value function associated to the cost (\ref{eq_cost}) with $q$ and $r$.
\end{itemize}
Hence, the control law (\ref{eq_OptCont}) makes the time derivative of the Lyapunov function $V$ nondecreasing and is also optimal with respect to cost defined from $q$ and $r$.
Note however that we get a weak Lyapunov function, i.e.,   a proper positive definite function whose derivative in direction of the vector field describing \eqref{eq_SystemTransfOrbit} is negative semi-definite.   Nevertheless, following \cite{KellettPraly_NOLCOS_04_ContThrustOrbTransfer}, it can be shown that employing this Lyapunov function in combination with LaSalle invariance principle, global asymptotic stabilization of the origin of the system (\ref{eq_SystemTransfOrbit}) with the control law (\ref{eq_OptCont}) is obtained.
\end{proof}

\section{Conclusion}

In this article we have developed a theory for constructing control laws having a predetermined local behavior.
In a first step, we showed that this problem can be rewritten as an equivalent problem in terms of control Lyapunov functions.
In a second step we have demonstrated that when the local behavior comes from an (LQ) optimal approach, we can characterize a cost with specific local approximation that can be minimized.
Finally, we have introduced two classes of system for which we know how to build these locally optimal control laws.

All this theory has been illustrated on the problem of orbital transfer.

\if\longversion1
\bibliographystyle{plain}
\fi
%
\bibliography{BibVA}

\end{document}